\documentclass[12pt]{birkmult}
\usepackage{amscd}
\usepackage{bbm}
\usepackage{dsfont}
\usepackage{enumerate}
\usepackage{latexsym}
\usepackage{srcltx}
\usepackage{amsfonts}
\usepackage{amssymb}
\usepackage{datetime}
\usepackage{setspace}
\usepackage{enumerate}
\usepackage{amsthm}
\usepackage{graphicx}
\usepackage{epstopdf}

\newtheorem{theorem}{Theorem}[section]

\newtheorem{lemma}[theorem]{Lemma}

\theoremstyle{definition}

\newtheorem{definition}[theorem]{Definition}

\newtheorem{remark} [theorem] {Remark}

\begin{document}

\title{Dedekind complete and order continuous Banach $C(K)$-modules.}

\author{Arkady Kitover}

\address{Department of Mathematics, Community College of Philadelphia, 1700 Spring Garden St., Philadelphia, PA, USA}

\email{akitover@ccp.edu}

\author{Mehmet Orhon}

\address{Department of Mathematics and Statistics, University of New Hampshire, Durham, NH, 03824}

\email{mo@unh.edu}

\dedicatory{ This paper is dedicated to our dear friend, Professor Ben de Pagter, on the occasion of his 65th birthday.}

\begin{abstract} We extend the notions of Dedekind complete and $\sigma$-Dedekind complete Banach lattices to Banach $C(K)$-modules. As our main result we prove for these modules an analogue of Lozanovsky's well known characterization of Banach lattices with order continuous norm.

\end{abstract}



\date{\today}

\maketitle

\markboth{Arkady Kitover and Mehmet Orhon}{Dedekind complete and order continuous Banach $C(K)$-modules.}

\section{Introduction}
This paper continues the investigation of the properties of finitely generated Banach $C(K)$-modules (see the definition preceding Theorem~\ref{t4}) undertaken by the authors in the papers~\cite{KO1, KO2, KO3}. There is good reason to consider such modules as the nearest relatives of Banach Lattices. Indeed, in the above mentioned papers the authors proved that the well known criteria of reflexivity, weak sequential completeness, and dual Radon - Nikodym property established for Banach Lattices remain valid for finitely generated Banach $C(K)$-modules, while, as it is also well known, for arbitrary Banach spaces this is not the case.

It provides some hope that it is possible to develop a meaningful theory of finitely generated Banach $C(K)$-modules parallel to the one of Banach lattices, and the current paper can be considered as another small step in this direction.

It is a well known result of Lozanovsky~\cite{Lo} that a Dedekind complete Banach lattice has order continuous norm (in short: is order continuous) if and only if it does not contain a copy of $l^\infty$. It is sufficient actually to require that the Banach lattice is $\sigma$-Dedekind complete (see e.g.~\cite{Wn}), but some condition of this type is, of course, necessary, because e.g. $C[0,1]$, being separable, does not contain a copy of $l^\infty$, but the standard norm on it is not order continuous.

Now, the following questions arise.
\begin{enumerate}
  \item  What is a natural analogue of Dedekind completeness for Banach $C(K)$-modules?
  \item The same question for $\sigma$-Dedekind completeness.
  \item What should be considered as an analog of order continuity for Banach $C(K)$-modules?
  \item Does an analog of Lozanovsky's result remain valid for \textbf{finitely generated} Banach $C(K)$-modules?
\end{enumerate}

As the reader will see, the answer to the first question is provided by the well known notion of Kaplansky module (see Definition~\ref{d5}).

The second question is more involved, and our answer to it makes use of some deep ideas of the late A.I. Veksler~\cite{Ve}

An answer to the third question, as well as a positive answer to question four (see Theorems~\ref{t3} and~\ref{t4}) represent the main results of the current paper.

After this short introduction we will proceed with some basic definitions, concepts, and notations.

\section{Preliminaries}

Let $K$ be a compact Hausdorff space and $C(K)$ be the Banach algebra of all complex (or real)-valued continuous functions on $K$. Let $X$ be a Banach space over the field $\mathds{C}$ of complex numbers or over the field $\mathds{R}$ of real numbers. Let $m$ be a unital bounded algebra homomorphism of $C(K)$ into the algebra $L(X)$ of all bounded linear operators on $X$. The triple $(C(K),m,X)$ is called a Banach $C(K)$-module.

Given a Banach $C(K)$-module we can define an equivalent norm on $X$,
$$ \|x\|^\prime = sup\{\|ax\|,a \in C(K), \|a\| \leq 1\}.$$
With respect to this norm the homomorphism $m$ is a contraction. Since $m$ is an algebra homomorphism its kernel, $\ker{m}$ is a closed ideal in the algebra $C(K)$. By replacing $C(K)$ with $C(K^\prime) \cong C(K)/\ker{m}$ where
$K^\prime = \{k \in K: a(k)=0 \; \forall a \in \ker{m}\}$ we can always assume that $m$ is one-to-one. The following lemma was proved in~\cite[Lemma 2]{HO1}.

\begin{lemma} \label{l1} Let $m: C(K) \rightarrow L(X)$ be a contractive homomorphism. Then

\noindent $ 1. \; a,b \in C(K), |a| \leq |b| \Rightarrow \|m(a)x\| \leq \|m(b)x\|$ for any $x \in X.$

 \noindent 2. If $m$ is one-to-one then it is an isometry.

\end{lemma}

\begin{remark} \label{r1} We would like to emphasize that while in the statement of our main result we will suppose only that $m$ is a bounded homomorphism, in its proof, in virtue of Lemma~\ref{l1} we will assume that it is an isometry.

\end{remark}

\begin{definition} \label{d1} Let $x \in X$. The cyclic subspace $X(x)$ in $X$ is defined as

$$ X(x) = cl \{m(a)x, \; a \in C(K)\}, $$
where $cl$ denotes the closure of a set.

A Banach $C(K)$-module is called a \textbf{cyclic Banach space} if there is an $x \in X$ such that $X = X(x)$.

\end{definition}

The following important and well known lemma is a consequence of more general results proved in~\cite{HO2} and~\cite{AAK}. Special cases were considered earlier by Veksler~\cite{Ve} and Schaefer~\cite{Sch}. A direct proof is in~\cite[Lemma 2]{OO}

\begin{lemma} \label{l2} Let $m: C(K) \rightarrow L(X)$ be a contractive unital homomorphism. Let $X = X(x_0)$ for some $x_0 \in X$, i.e. $X$ is a cyclic Banach space with a cyclic vector $x_0$. Then
\begin{enumerate}
  \item $X$ can be represented as a Banach lattice with quasi-interior point $x_0$.
  \item The cone of the positive elements in $X$ is the closure in $X$ of the set $\{m(a)x_0: \; a \in C(K), a \geq 0\}$.
  \item The center Z(X) of the Banach lattice $X$ can be identified with the closure of $m(C(K))$ in the weak operator topology on $L(X)$.
  \item The unit ball of $Z(X)$ is the closure of the unit ball of $m(C(K))$ in the weak operator topology.
  \item If $x$ is a quasi-interior point in the Banach lattice $X$ then for the order ideal $I_x$ generated by $x$ we have $I_x = Z(X)x$.
\end{enumerate}

\end{lemma}

Another important tool needed for our results is the following lemma~\cite[Lemma 4(3), equivalence $(b) \Leftrightarrow (c)$]{HO1}.

\begin{lemma} \label{l3} Let $m: C(K) \rightarrow L(X)$ be a bounded unital algebra homomorphism. Then the following statements are equivalent.

\begin{enumerate}
  \item For each $x \in X$ the map $a \rightarrow m(a)x, a \in C(K)$, is compact, provided that $C(K)$ is endowed with its norm topology and $X$ with its weak topology.
  \item  The homomorphism $m$ can be extended uniquely to an algebra homomorphism $\hat{m} : C(K)^{\prime \prime} \rightarrow L(X) $
   which is continuous provided that $C(K)^{\prime \prime}$ is endowed with the weak-star topology and $L(X)$ with the weak operator topology.
\end{enumerate}

\end{lemma}

\begin{remark} \label{r2} (a) Note that $C(K)^{\prime \prime} \cong C(S) $ where $S$ is a hyperstonian compact space. The proof of Lemma~\ref{l3} given in~\cite{HO1} requires the use of Arens extension of the $C(K)$-module multiplication on $X$ to a $C(K)^{\prime \prime}$-module multiplication on $X^\prime$ and $X^{\prime \prime}$. Since we will not need the use of this result other than in the above form, we will not introduce Arens extensions here. However, we encourage the interested reader to look up the straightforward exposition and the proof in~\cite[ pp. 75 - 76]{HO1}.

(b) The property (1) in Lemma~\ref{l3} is called the \textbf{weakly compact action} of $m(C(K))$ on $X$. We will use this term in the sequel.

\end{remark}

\section{Dedekind complete and $\sigma$-Dedekind complete cyclic Banach spaces}

In this section we will characterize cyclic Banach spaces which, when represented as
Banach lattices in accordance with Lemma~\ref{l2} are Dedekind complete or $\sigma$-Dedekind complete. A characterization of $\sigma$-Dedekind complete cyclic Banach spaces was obtained by Veksler in~\cite{Ve}. Here we will simplify Veksler's proof by using Lemmas~\ref{l1} and~\ref{l2}. But we need to emphasize that the main idea is still the one used by Veksler.

For the remainder of this section we will assume that the compact space $K$ is \textbf{totally disconnected}. We will denote by $\mathcal{B}$ the Boolean algebra of all the idempotents in $C(K)$. Clearly, $\mathcal{B}$ consists of characteristic functions of the clopen subsets of $K$. The identically one function $\mathbf{1}$ and the identically zero function $\mathbf{0}$ correspond to the identity and the zero of the Boolean algebra $\mathcal{B}$. Let $m : C(K) \rightarrow L(X)$ be a bounded unital algebra homomorphism, notice that the closed subset $K^\prime$ of $K$ defined in Section 2 as the set of common zeros of functions from $\ker{m}$, is also a totally disconnected compact Hausdorff space with the relative topology inherited from $K$.

\begin{definition} \label{d2} Let $x \in X$. An idempotent $e_x \in \mathcal{B}$ is called the \textbf{carrier projection} of $x$ if $m(e_x) x = x$ and
$$ e\in \mathcal{B}, \; m(e)x=x \Rightarrow e_x \leq e \; \mathrm{in} \; C(K). $$

\end{definition}

\begin{remark} \label{r3} Such a projection is called by Veksler the support of $x$ (see~\cite{Ve}). It is worth noticing that when $m(\mathcal{B})$ is a Bade complete Boolean algebra of projections on $X$ (see Definition~\ref{d6} below) the above definition coincides with the Bade's definition of carrier projection in~\cite{DSIII}.

\end{remark}

\begin{lemma} \label{l4} Let $x \in X$ and $e_x$ be the carrier projection of $x$. Then for any $a \in C(K)$ , $m(a)x=0$ if and only if $ae_x =0$.

\end{lemma}

\begin{proof} Suppose $ae_x = 0$. Then $m(a)x = m(a)m(e_x)x = m(ae_x)x=0$.

Conversely, initially suppose that for some $e \in \mathcal{B}$, $m(e)x = 0$. Then $m((1 - e))x=x$ and therefore $e_x \leq 1- e$. It follows that $ee_x = 0$. Now assume that for some non-negative $a \in C(K)$ we have $m(a)x = 0$ and that for some $t \in K$, $a(t) > 0$. Then, since $K$ is totally disconnected,  there are some  $\varepsilon > 0$ and $e \in \mathcal{B}$ such that $e(t) = 1$ and $0 \leq \varepsilon e \leq a$. Then by Lemma~\ref{l1} (1), $\varepsilon \|m(e)x\| \leq \|m(a)x\| = 0$. Therefore $ee_x = 0$ and $e_x(t) = 0$, hence $ae = 0$.

Finally, if $m(a)x = 0$ for some $a \in C(K)$ then applying again Lemma~\ref{l1} (1) we see that $m(|a|)x = 0$ hence $|a|e_x =0$, and therefore $ae_x = 0$.

\end{proof}

\begin{definition} \label{d4} We will say that a Banach $C(K)$-module $X$ is a \textbf{Veksler module} if any $x \in X \setminus \{0\}$ has a carrier projection $e_x \in C(K)$.

\end{definition}

\begin{remark} \label{r4}  Cyclic Banach spaces with the property stated in Definition~\ref{d4} were introduced by Veksler in~\cite{Ve}. They were also considered by Rall~\cite{Ra} who called the corresponding Boolean algebra $\mathcal{B}$ $\tau$-complete.

\end{remark}

Assume that $X$ is a Veksler module. For every $x \in X \setminus \{0\}$ we will denote the clopen support of $e_x$ in $K$ by $K_x$. Clearly, $K_x$ is the Stone representation space of the Boolean algebra $e_x\mathcal{B}$.

Let us recall that a compact Hausdorff space $K$ is called quasi-Stonian if it is basically disconnected, i.e. the closure of every open $G_\delta$ set in $K$ is open. It is also worth recalling that the following properties are equivalent:
\begin{enumerate}
  \item $K$ is quasi-Stonian.
  \item $C(K)$ is $\sigma$-Dedekind complete.
  \item Every non-negative sequence bounded from above in $C(K)$ has a supremum in $C(K)$.
  \item Every principal band in $C(K)$ is a projection band.
\end{enumerate}

\begin{lemma} \label{l5} Let $X$ be a Veksler module with respect to the Boolean algebra $\mathcal{B}$. Then for every $x \in X \setminus \{0\}$, the Banach algebra $C(K_x)$ is a Veksler module with respect to the Boolean algebra $e_x \mathcal{B}$. Moreover, the compact space $K_x$ is quasi-Stonian.

\end{lemma}

\begin{proof} Let $x\in X$ and let $e_x \in \mathcal{B}$ be the corresponding carrier projection. Let $a \in C(K_x)$. We extend $a$ as a continuous function on $K$ by letting $a \equiv 0$ on $K \setminus K_x$. We will identify $a$ and its extension on $K$. We claim that $e_{m(a)x}$ is the carrier projection for $a$ in $e_x \mathcal{B}$. Note that $m(1-e_x)m(a)x=0$. Therefore $(1-e_x)e_{m(a)x}$=0 and $e_{m(a)x} \in e_x\mathcal{B}$. Now suppose that for some $b \in C(K_x)$, $ba = 0$. Then $m(b)m(a)x=0$, and consequently $be_{m(a)x} = 0$. Conversely, suppose $be_{m(a)x} = 0$. Then $m(b)m(a)x=0$ and therefore $bae_x=0$. Hence $e_{m(a)x}$ is the carrier projection of $a \in C(K_x)$. This also means that $\{a\}^\perp = (1 - e_{m(a)x})C(K_x)$. Therefore the principal band $\{a\}^{\perp \perp} = e_{m(a)x}C(K_x)$ is a projection band and the space $K_x$ is quasi-Stonian.

\end{proof}

The critical property of cyclic subspaces of a Veksler module is stated in the next lemma due to Veksler~\cite{Ve}. For the sake of completeness we will give a simplified proof of this result.

\begin{lemma} \label{l6} Let $X$ be a cyclic Veksler module with a cyclic vector $x_0$. If $x_n \rightarrow x$, $y_n \rightarrow y$ and for each $n$, $e_{x_n}e_{y_n}=0$ then $e_xe_y=0$.

\end{lemma}

\begin{proof}Because $m$ is an isometry throughout this proof we will identify $a \in C(K)$ and $m(a) \in L(X)$. We assume that by Lemma~\ref{l2}, $X$ has been represented as a Banach lattice with quasi interior point $x_0$. Since $C(K) \subseteq Z(X)$, we have for some $a \in C(K)$ and $x \in X$ that $ax =0$ if and only if $a|x|=0$. Therefore the carrier projections of $x$ and $|x|$ are the same. So it is sufficient to consider non-negative sequences $\{x_n\}$ and $\{y_n\}$. Also, since $x_n \rightarrow x$ implies $e_x x_n \rightarrow e_x x=x$, we can assume that $e_{x_n} \leq e_x$ and $e_{y_n} \leq e_y$. Next we notice that $e_{x_n}x \rightarrow x$ and $e_{y_n}y \rightarrow y$. Indeed,
$$ \|e_{x_n}x-x\| \leq \|e_{x_n}(x-x_n)\| +\|x_n - x\| \leq 2\|x_n - x\|.$$
We may assume that $\|x-e_{x_n}x\| < \frac{1}{2^n}$ and $\|y-e_{y_n}y\| < \frac{1}{2^n}$ for $n \in \mathds{N}$.  For any $m \in \mathds{N}$ consider the decreasing sequence
$x_n^m = (e_{x_m}e_{x_{m+1}} \ldots e_{x_n})x $ where $n \geq m$. Let $m \leq n$ and $p \in \mathds{N}$. Then
$$ x_n^m - x_{n+p}^m = (e_{x_m}e_{x_{m+1}} \ldots e_{x_n})[(1-e_{x_{n+1}}) + e_{x_{n+1}}(1- e_{x_{n+2}}) + \ldots $$

$$ + e_{x_{n+1}}e_{x_{n+2}} \ldots e_{x_{n+p-1}}(1 - e_{x_{n+p}})]x. $$
Therefore $\|x_n^m - x_n^{m+p}\| \leq \frac{1}{2^n}$ and we can assume that $x_n^m \mathop \rightarrow \limits_{n \to \infty} x^m \in X $. Similarly $y_n^m \rightarrow y^m \in X$. Since both sequences are decreasing, for all $n \geq m$ we have $e_{x^m} \leq e_{x_n}$ and
$e_{y^m} \leq e_{y_n}$.

It is clear that $\{x^m\}$ and $\{y^m\}$ are increasing sequences and that $x$ and $y$ are upper bounds for each sequence, respectively. Next consider for any fixed $m$ the increasing sequence $\{x - x_n^m\}_{n \geq m}$. We have
$$ x - x_n^m = [(1-e_{x_{m}}) + e_{x_{m}}(1- e_{x_{m+1}}) + e_{x_m}e_{x_{m+1}}(1- e_{x_{m+2}})+ \ldots $$

$$ + e_{x_{m}}e_{x_{m+1}} \ldots e_{x_{n-1}}(1 - e_{x_{n}})]x.$$
It follows easily that $\|x-x_n^m\| < \frac{1}{2^{m-1}}$. Hence, passing to the limit, we have
$\|x - x^m\| \leq \frac{1}{2^{m-1}}$. Therefore $x^m \rightarrow x$ and $y^m \rightarrow y$. Since $e_{y^m}e_{x^n} = 0$ for all $n \geq m$, we have $e_{y^m}x^n = 0$, $n \geq m$. Hence,
$e_{y^m}x = 0$ and therefore $e_{y^m}e_x=0$ for all $m$. Then $e_x y^m = 0$ for all $m$, hence $e_x y=0$. It follows that $e_xe_y = 0$.

\end{proof}

\begin{lemma} \label{l7} (Veksler) A cyclic Veksler module represented as a Banach lattice is $\sigma$-Dedekind complete.

\end{lemma}

\begin{proof} It follows from Lemmas~\ref{l6} and~\ref{l5} that $m(C(K))$ is weak-operator closed. To see this, note that by Lemma~\ref{l2}, the weak operator closure of $C(K)$ is $Z(X) = C(\hat{K})$. Since $\mathcal{B} \subset C(K) \subseteq C(\hat{K})$ it is sufficient to prove that $\mathcal{B}$ separates the points of $\hat{K}$. Let $s, t$ be two distinct points in $\hat{K}$. Then there is $\Phi \in C(\hat{K})$ such that $-1 \leq \Phi \leq 1$, $\Phi(s)=1$, and $\Phi(t) = -1$. Let $x_0$ be a cyclic vector in $X$. Without loss of generality we will assume that the carrier projection of $x_0$ is $\mathbf{1} \in  \mathcal{B}$. There is a sequence $\{a_n\}$ in $C(K)$ such that $\|a_n x_0 - \Phi x_0\| \rightarrow 0$. Since $|a_n^+ - \Phi^+| \leq |a_n - \Phi|$ and $|a_n^- - \Phi^-| \leq |a_n - \Phi|$, where the lattice operations are considered in $C(\hat{K}) = Z(X)$, we have  $\|(a_n^+ - \Phi^+)x_0\| \leq \|(a_n - \Phi)x_0\|$ and $\|(a_n^- - \Phi^-)x_0\| \leq \|(a_n - \Phi)x_0\|$. Thus $a_n^+ x_0 \rightarrow \Phi^+x_0$ and $a_n^- x_0 \rightarrow \Phi^-x_0$ in norm in $X$. Let $e_n, e_n^\prime \in \mathcal{B}$ be the carrier projections of $a_n^+$ and $a_n^-$ in $C(K)$, respectively (they exist because $K$ is quasi-Stonian). Then clearly $e_n^\prime \leq 1 - e_n$. By Lemma~\ref{l5}, $e_{a_n^+x_0}=e_n$ and $e_{a_n^-x_0}=e_n^\prime$.
Let $e$ and $e^\prime$ be carrier projections in $\mathcal{B}$ of $\Phi^+x_0$ and $\Phi^-x_0$, respectively. By Lemma~\ref{l6}, $ee^\prime = 0$. So $e \Phi^-x_0 = 0$ and, because $x_0$ is a $C(K)$-cyclic vector, $e\Phi^-=0$, as an operator in $Z(X)$, hence $e(t)=0$. Moreover, $ (1-e)\Phi^+x_0  =0$ hence $(1-e)\Phi^+=0$ and $e(s)=1$. Thus, $K = \hat{K}$ and $Z(X) = C(K)$.

Assume that $x_n \in X_+$ and $x_n \leq y \in X, n \in \mathds{N}$. Let $z=x_0 \vee y$, then $z$ is a quasi-interior element in $X$ and $x_n \in I_z$ where $I_z$ is the order ideal in $X$ generated by the quasi-interior element $z$. Then, by Lemma~\ref{l2}(5) (actually, by the Krein-Kakutani theorem),  $I_z = Z(X)z$. Hence there are $a_n \in C(K) = Z(X)$ such that $0 \leq a_n \leq 1$ and $x_n = a_n z$. Recall that $K$ is quasi-Stonian, hence $C(K)$ is $\sigma$-Dedekind complete. Let $a = \sup \limits_{n \in \mathds{N}} a_n \in C(K)$. It is immediate to see that $az = \sup \limits_{n \in \mathds{N}} a_n z \in X$.

\end{proof}

It is well known that if $E$ is a Dedekind complete vector lattice its center $Z(E) = C(K)$ is also Dedekind complete, hence $K$ is a Stonian (extremally disconnected) compact space. An analog of this property for Banach $C(K)$-modules was first considered by Kaplansky~\cite{Ka}.

\begin{definition} \label{d5} A Banach $C(K)$-module $X$ is called a \textbf{Kaplansky module} if it satisfies the following two conditions.
\begin{enumerate}
  \item The compact space $K$ is Stonian.
  \item For any $x \in X$ and for any non-negative set $\{a_\alpha\}$ bounded above  in $C(K)$ the following implication holds
      $$ a_\alpha x = 0, \; \text{for all} \; \alpha \Rightarrow ax=0, \; \mathrm{where} \; a = \sup a_\alpha.$$
\end{enumerate}

\end{definition}

It is easy to see that any Kaplansky module is also a Veksler module and that any Dedekind complete Banach lattice $X$ is a Kaplansky module over its center $Z(X)$. The lemma below shows that the converse of the last statement is true for cyclic Kaplansky modules.

\begin{lemma} \label{l8} Let $X$ be a cyclic Kaplansky module over $C(K)$ and let $x_0 \in X$ be a cyclic element of $X$. Then when represented as a Banach lattice with the quasi-interior point $x_0$, $X$ is Dedekind complete.

\end{lemma}

\begin{proof} Because, as a Kaplansky module, $X$ is also a Veksler module we can apply Lemma~\ref{l7} to conclude that $Z(X) = C(K)$. Note that $C(K)$ is Dedekind complete. Then repeat the argument from the proof of Lemma~\ref{l7} to see that any subset of $X_+$ that is bounded above has a supremum in $X$.

\end{proof}

\section{Order continuous Banach $C(K)$-modules}

In this section we will extend to Banach $C(K)$-modules the following classic result of Lozanovsky~\cite{Lo}.

\begin{theorem} \label{t1} Let $E$ be a $\sigma$-Dedekind complete Banach lattice. The following conditions are equivalent.
\begin{enumerate}
  \item The original lattice norm on $E$ is order continuous.
  \item $E$ does not contain $\ell^\infty$ as a closed subspace.
  \item $E$ does not contain $\ell^\infty$ as a closed sublattice.
\end{enumerate}

\end{theorem}

Throughout this section we assume that $X$ is a Banach space either over $\mathds{R}$ or over $\mathds{C}$, $K$ is a compact Hausdorff space, and $m: C(K) \rightarrow L(X)$ is a bounded injective unital algebra homomorphism (and therefore an isometry).

\begin{definition} \label{d6} Let $K$ be a Stonian compact Hausdorff space and $\mathcal{B}$ is the Boolean algebra of all idempotents in $C(K)$. We say that $\mathcal{B}$ is a \textbf{Bade-complete Boolean algebra of projections} on $X$ if for any increasing net $\{e_\alpha\}$ in $\mathcal{B}$ and for any $x \in X$ we have $\|(e - e_\alpha)x\| \rightarrow 0$, where $e = \sup{e_\alpha}$.

\end{definition}

\begin{remark} \label{r5} The above definition is equivalent to the definition of a complete Boolean algebra of projections on a Banach space as given by Bade~\cite{DSIII}.

\end{remark}

We start with our main result concerning cyclic Banach spaces.

\begin{theorem} \label{t2} Let $X$ be a cyclic Banach $C(K)$-module, $x_0$ be a cyclic vector in $X$, and let $\mathcal{B}$ be the Boolean algebra of all idempotents in $C(K)$ . The following conditions are equivalent.
\begin{enumerate}
  \item The compact space $K$ is totally disconnected, $X$ is a Veksler module, and $X$ does not contain a copy of $\ell^\infty$.
  \item The compact space $K$ is Stonian, $X$ is a Kaplansky module, and $X$ does not contain a copy of $\ell^\infty$.
  \item $m(C(K))$ is weak operator closed, and $X$, when represented as a Banach lattice, has order continuous norm.
  \item The compact space $K$ is Stonian and $\mathcal{B}$ is a Bade complete Boolean algebra of projections on $X$.
  \item The compact space $K$ is hyperstonian and $m$ is ($w^\star$, weak-operator)-continuous.
\end{enumerate}

\end{theorem}

\begin{proof}  $(2) \Rightarrow (1)$. This implication follows from the fact that a Kaplansky module is a Veksler module.

$(1) \Rightarrow (3)$.  Note that when $X$ is represented as a Banach lattice with a quasi-interior point $x_0$, it is $\sigma$-Dedekind complete and $m(C(K)) = Z(X)$ ( see Lemma~\ref{l7} and its proof). Therefore $m(C(K))$ is weak operator closed. Next, since $X$ does not contain any copy of $\ell^\infty$, $X$ has order continuous norm by Theorem~\ref{t1}.

 $(3) \Rightarrow (2)$. It is well known that $X$ has order continuous norm implies that $X$ is Dedekind complete and by Theorem~\ref{t1} it does not contain a copy of $\ell^\infty$.

  $(3) \Rightarrow (4)$. Since $m(C(K))$ is weak operator closed, by Lemma~\ref{l2} (3) we have $m(C(K)) = Z(X)$. As we have already noticed, $X$ is Dedekind complete. Therefore $Z(X)$ is Dedekind complete and $K$ is Stonian. Moreover, if $\{e_\alpha\}$ is an increasing net in $\mathcal{B}$ and $e = \sup{e_\alpha}$ then for any non-negative $x \in X$ we have $\sup{e_\alpha x} = ex$. Therefore, $\|e_\alpha x - ex\| \rightarrow 0$. It follows that $\mathcal{B}$ is a Bade complete Boolean algebra of projections on $X$.

   $(5) \Rightarrow (3)$. The fact that $K$ is hyperstonian implies that $C(K)$ is a dual Banach lattice. Let us denote its predual by $C(K)_\star$. The definition of order in $C(K)$ by means of its predual implies that whenever $0 \leq a_\alpha$ is an increasing net and $\sup{a_\alpha} = a$ in $C(K)$ then $a_\alpha \rightarrow a$ in the $w^\star$ topology in $C(K)$. Since $m$ is ($w^\star$, weak-operator)-continuous , we have $a_\alpha x \rightarrow ax $ in the weak topology in $X$. In particular if for every $\alpha$, $a_\alpha x = 0$ then $ax = 0$ and therefore $X$ is a Kaplansky module. Hence, by Lemma~\ref{l8}, when represented as a Banach lattice, $X$ is Dedekind complete and $m(C(K)) = Z(X)$. Next assume that $\{x_\alpha\}$ is an increasing net bounded above with non-negative elements in $X$ and that $x = \sup{x_\alpha}$. Then there is an increasing net $\{a_\alpha\}$ of non-negative elements in $C(K)$ such that $x_\alpha = a_\alpha x$. It is immediate to see that $\sup{a_\alpha} = e_x$, the carrier projection of $x$. Then the net $\{a_\alpha\}$ converges to $e_x$ in the $w^\star$ topology in $C(K)$ and therefore, $\{x_\alpha\}$ converges to $x$ in the weak topology in $X$. Duality implies that there is a sequence $\{y_n\}$ in $X$ such that every $y_n$ is a convex combination of elements from the net $\{x_\alpha\}$ and $\|y_n - x \| \rightarrow 0$. Since $\{x_\alpha\}$ is an increasing net of non-negative elements, it follows that the net $\{x_\alpha\}$ converges to $x$ in norm, hence the norm on $X$ is order continuous.

   $(4) \Rightarrow (5)$. Suppose $x \in X$, $x^\prime \in X^\prime$ and define $\mu_{x,x^\prime} \in C(K)^\prime$ as
   $$ \mu_{x,x^\prime}(a) = x^\prime (ax), \; a \in C(K). $$
We want to show that each such functional is order continuous on $C(K)$. Because the set of these functionals is clearly total on $C(K)$, it would follow that $K$ is hyperstonian. It is sufficient to show that for any closed nowhere dense subset $D$ of $K$ we have $\mu_{x,x^\prime}(D) = 0$ where we identify $\mu_{x,x^\prime}$ with the corresponding finite regular Borel measure on $X$. Initially assume that $\mu_{x,x^\prime} \geq 0$. Let us fix $D$ and let us consider the collection $\{\tau \}$ of clopen subsets of $K$ disjoint from $D$ and ordered by inclusion. Let $e_\tau \in \mathcal{B}$ be the characteristic function of the subset $\tau$. Then $\{e_\tau\}$ is an increasing net in $\mathcal{B}$. Because the union of all the sets in the collection $\{\tau\}$ is $K \setminus D$, an open dense subset of $K$, we have $\sup{e_\tau} = 1$. Recalling that $\mathcal{B}$ is Bade complete we see that $\|(1 - e_\tau)x\| \rightarrow 0$. From the inequalities
$$ 0 \leq \mu_{x,x^\prime}(D) \leq \mu_{x,x^\prime}(K \setminus \tau) = \mu_{x,x^\prime}(1 - e_\tau) = x^\prime ((1 - e_\tau)x) $$
we conclude that $\mu_{x,x^\prime}(D) = 0$.

Now denote by $X_R$ the Banach space $X$ considered as a real Banach space and by $X_R^\prime$ the space of all real valued bounded linear functionals on $X_R$. We claim that for any $x^\prime \in X_R^\prime$ and for any $x \in X$ there exists $e_+ \in \mathcal{B}$ such that $\mu_{e_+x,x^\prime} = \mu_{x,x^\prime}^+$ and $-\mu_{(1-e_+)x,x^\prime} = \mu_{x,x^\prime}^-$. Given $\mu_{x,x^\prime}$ let us, without loss of generality, assume that both $\mu_{x,x^\prime}^+$ and $\mu_{x,x^\prime}^-$ are not equal to zero. We can assume, without loss of generality, that for some $e \in \mathcal{B}$ we have $\mu_{x,x^\prime}(e) > 0$ and that the set $\mathcal{N} = \{e^\prime \in e\mathcal{B}: \; \mu_{x,x^\prime}(e^\prime) < 0\}$ is not empty. By Zorn's lemma there is a maximal collection $\Omega$ of pairwise disjoint elements of $\mathcal{N}$. Let $\{e_\alpha\}$ be an increasing net in $\mathcal{B}$ formed by finite sums of elements of $\Omega$. Let $e^\prime \in \mathcal{B}$ be the supremum of $\{e_\alpha\}$. Since $\mathcal{B}$ is Bade complete on $X$, we have $\|(e^\prime - e_\alpha)x\| \rightarrow 0$. Also, since for every $\alpha$, $\mu_{x,x^\prime}(e_\alpha) < 0$ and the net $\{\mu_{x,x^\prime}(e_\alpha)\}$ is decreasing, we have
$$ \mu_{x,x^\prime}(e^\prime) = \lim{\mu_{x,x^\prime}(e_\alpha)} < 0. $$
Then for any projection $p \in (e - e^\prime)\mathcal{B}$ we have $\mu_{x,x^\prime}(p) \geq 0$. Indeed, otherwise we arrive at a contradiction with the maximality of $\Omega$. Thus we have proved that the set
$$ \mathcal{P} = \{e \in \mathcal{B} :  e^\prime \in e\mathcal{B} \Rightarrow \mu_{x,x^\prime}(e^\prime) \geq 0 \} $$
is not empty. Applying Zorn's lemma again, we can find a maximal collection $\mathcal{E}$ of pairwise disjoint elements in $\mathcal{P}$. Let $e_+ = \sup \limits_{e \in \mathcal{E}} e \in \mathcal{B}$. We claim that $e_+ \in \mathcal{P}$. Indeed, let $e^\prime \in \mathcal{B}$. Then for any $e \in \mathcal{E}$ we have $\mu_{x,x^\prime}(e^\prime e) \geq 0$. Bade completeness of $\mathcal{B}$ on $X$ implies that $\mu_{x,x^\prime}(e^\prime e_+) \geq 0$, hence $e_+ \in \mathcal{P}$. It is easy to conclude from the maximality of $\mathcal{E}$ that for any idempotent $e \in (1 - e_+)\mathcal{B}$, we have $\mu_{x,x^\prime}(e) \leq 0$. It follows that both $\mu_{e_+x,x^\prime}$ and $-\mu_{(1-e_+)x,x^\prime}$ are positive linear functionals on $C_R(K)$ (the space of all real-valued continuous functions on $K$). Hence, as we have already proved, both of them, and therefore their difference $\mu_{x,x^\prime}$ are order continuous on $C_R(K)$. Since the map $m$ is one-to-one, it is clear the set of functionals $\{\mu_{x,x^\prime} : x \in X_R, x^\prime \in X_R^\prime\}$ is total on $C_R(K)$. Therefore $K$ is hyperstonian and $C(K)$ is a dual Banach lattice. Since each linear functional in $X^\prime$ is a linear combination over $\mathds{C}$  of functionals from $X_R^\prime$ every functional $\mu_{x,x^\prime}, x \in X, x^\prime \in X^\prime$ is order continuous on $C(K)$ and therefore belongs to its predual $C(K)_\star$. Therefore if $a_\alpha \rightarrow a$ in $w^\star$-topology in $C(K)$ then $m(a_\alpha) \rightarrow m(a)$ in the weak operator topology in $L(X)$.
\end{proof}

We are now ready to state and prove our main result for general Banach $C(K)$-modules. In connection with this we would like to make the following remark.

\begin{remark} \label{r6} (a) In condition (3) of Theorem~\ref{t2} we cannot dispense with the requirement that $m(C(K))$ is weak-operator closed. Indeed, $c_0$ is a cyclic Banach $c$-module and has order continuous norm. But $Z(c_0) = w-cl(m(c_0)) = \ell^\infty \neq m(c_0)$.

(b) As the reader will see, in our next result, Theorem~\ref{t3}, where we consider general Banach $C(K)$-modules, we assume from the very beginning that $m(C(K))$ is weak-operator closed. The following example illustrates that we have to be careful with  regard to this condition. We consider $c$ as a Banach $c$-module; then $m(c) = Z(c)$ is closed in the weak operator topology in $L(c)$ but when we restrict $m$ to the cyclic subspace $c_0$ as we have just seen $m(c)$ is not weak operator closed in $L(c_0)$.

\end{remark}

\begin{theorem} \label{t3} Let $m : C(K) \rightarrow L(X)$ be an injective bounded unital algebra homomorphism such that $m(C(K))$ is weak-operator closed. The following conditions are equivalent.
\begin{enumerate}
  \item
  \begin{enumerate}[(a)]
  \item $K$ is totally disconnected and
    \item $X$ is a Veksler module, and
    \item no cyclic subspace of $X$ contains a copy of $\ell^\infty$.
  \end{enumerate}
  \item
  \begin{enumerate}[(a)]
  \item $K$ is Stonian and
  \item $X$ is a Kaplansky module,
   and
    \item no cyclic subspace of $X$ contains a copy of $\ell^\infty$.
    \end{enumerate}
  \item Each cyclic subspace of $X$ when represented as a Banach lattice has order continuous norm.
  \item
  \begin{enumerate}[(a)]
    \item $K$ is Stonian and
    \item the set $\mathcal{B}$ of all idempotents in $C(K)$ is a Bade complete Boolean algebra of projections on $X$.
  \end{enumerate}
  \item
  \begin{enumerate}[(a)]
    \item $K$ is hyperstonian and
    \item the map $m$ is ($w^\star$, weak operator) continuous.
  \end{enumerate}
\end{enumerate}
 \end{theorem}

 \begin{proof} $(2) \Rightarrow (1)$. This implication is trivial.

 $(1) \Rightarrow (3)$. It follows from (1) and Lemma~\ref{l7} that for any $x \in X$ the cyclic subspace $X(x)$ when represented as a Banach lattice is $\sigma$-Dedekind complete and $Z(X(x)) = C(K_x)$ where $K_x$ is the support of $e_x$ in $K$. Because $X(x)$  does not contain a copy of $\ell^\infty$ by Theorem~\ref{t1} ( $(2) \Rightarrow (1))$, the implication is proved.

 $(3) \Rightarrow (5)$. First we claim that (3) implies that $m(C(K))$ has weakly compact action on $X$. Let $x \in X$. Then $X(x)$ has order continuous norm when represented as a Banach lattice. Notice that for every $a \in C(K)$ the cyclic space $X(x)$ is invariant for the operator $m(a)$. Then the correspondence $a \rightarrow m(a)|X(x)$ defines the algebraic homomorphism $m_x : C(K) \rightarrow L(X(x))$. Let $C(K_x)\cong C(K)/\ker{m_x}$ and let $\dot{m}_x : C(K_x) \rightarrow L(X(x))$ be the unital injective algebra homomorphism induced by $m_x$. By Lemma~\ref{l2} the closure of $\dot{m}_x(C(K_x))$ in the weak operator topology in $L(X(x))$ can be identified with $Z(X(x))$. Let $L$ be the Stone representation space of $Z(X(x))$. Clearly $C(K_x)$ can be identified with a closed subalgebra of $C(L)$, the map $\dot{m}_x$ can be extended to an algebraic unital injective homomorphism $\tilde{m}_x : C(L) \rightarrow L(X(x))$, and $\tilde{m}_x(C(L)) = Z(X)$ hence $\tilde{m}_x(C(L))$ is weak operator closed in $L(X(x))$.  Applying the implication $(3) \Rightarrow (5)$ from Theorem~\ref{t2} we obtain that the Stone representation space $L$ of $Z(X(x))$ is hyperstonian and the embedding $\tilde{m}_x : C(L) \rightarrow L(X(x))$ is ($w^\star$, weak-operator)-continuous. Because $\tilde{m}_x(a)x = m(a)x, \; \textrm{when} \;  a \in C(K)$, we have that for any $x \in X$ the map $a \rightarrow m(a)x$ is weakly compact, and thus the claim is proved.

 Then, by Lemma~\ref{l3} there is a unique extension of $m$, $\hat{m} : C(K)^{\prime \prime} \rightarrow L(X)$ that is ($w^\star$, weak-operator)-continuous. Since $C(K)^{\prime \prime}$ is a dual $AM$-space its Stone representation space is hyperstonian. Because $\ker{\hat{m}}$ is a $w^\star$-closed ideal in $C(K)^{\prime \prime}$ we have that $C(K)^{\prime \prime}/\ker{\hat{m}} = C(S)$ where $S$ is a hyperstonian compact space. This means that as well as being an isometric unital algebra homomorphism the map $\hat{m} : C(S) \rightarrow L(X)$ is ($w^\star$, weak-operator) continuous. Therefore $X$ is a Kaplansky $C(S)$-module. We claim that this implies that $\hat{m}(C(S))$ is weak-operator closed in $L(X)$. For each $x \in X$, let $e_x$ denote the carrier projection of $x$ and $S_x$ the clopen subset of $S$ that is the support of $e_x$. Note that $C(S_x) = e_xC(S) = C(S)/(\ker{\hat{m}|X(x)})$.  Since  $\ker{\hat{m}|X(x)}$ is $w^\star$-closed, $C(S_x)$ is a dual Banach space and $S_x$ is hyperstonian. By Lemma~\ref{l8}, $X(x)$ when represented as a Banach lattice is Dedekind complete and $Z(X(x)) = C(S_x)$ is weak-operator closed in $L(X(x))$. Consider $T \in L(X)$ which is in the weak-operator closure of $\hat{m}(C(S))$. Then for any $x \in X$ the cyclic subspace $X(x)$ is $T$-invariant and $T|X(x) \in Z(X(x))$. Therefore for each $x \in X$ there is $a_x \in C(S_x) = e_xC(S)$ such that $T(x) = \hat{m}(a_x)x$. Then as proved (by a standard argument) in~\cite[Lemma 2]{HO2} there is $a \in C(S)$ such that $e_x a = a_x, x \in X$ and therefore $T(x) = \hat{m}(a)x, x \in X$. Thus $T \in \hat{m}(C(S))$ and as claimed $\hat{m}(C(S))$ is weak-operator closed. But $m(C(K))$ is weak-operator closed, $C(K)$ is $w^\star$-dense in C(S), and $\hat{m}$ is ($w^\star$, weak operator)- continuous. Hence $m(C(K))=\hat{m}(C(S))$ and $K$ is homeomorphic to $S$.

  $(5) \Rightarrow (4)$. Suppose that $\{e_\alpha\}$ is an increasing net in $\mathcal{B}$ such that $\sup{e_\alpha} = e \in \mathcal{B}$. Then $e_\alpha \rightarrow e$ in the $w^\star$ topology in $C(K)$. Let $x \in X$. Because $m$ is ($w^\star$, weak-operator)-continuous we have $e_\alpha x \rightarrow ex$ in the weak topology in $X$.
Duality implies that for a sequence $\{a_n\}$ of convex combinations of $e_\alpha$ we have
$\|ex - a_n x\| \rightarrow 0$. But the net $\{e_\alpha\}$ is increasing and therefore for a fixed $n$ and for all sufficiently large $\alpha$ we have $\|ex - e_\alpha x\| \leq \|ex - a_nx\|$. Hence $\|ex -  e_\alpha x \| \rightarrow 0 $ and $\mathcal{B}$ is a Bade complete Boolean algebra of projections on $X$.

$(4) \Rightarrow (3)$. It follows immediately from (4) that $X$ is a Kaplansky Banach $C(K)$-module. It remains to notice that for each $x \in X$ the Boolean algebra $e_x\mathcal{B}$ is Bade complete on $X(x)$ and apply the implication $(4) \Rightarrow (3)$ from Theorem~\ref{t2}.

$(4) \Rightarrow (2)$ We have already proved that (4) implies that $X$ is a Kaplansky module and that each cyclic subspace $X(x)$ when represented as a Banach lattice has order continuous norm. By Theorem~\ref{t1}, $X(x)$ does not contain any copy of $\ell^\infty$.
 \end{proof}

 Our next theorem relates to the case of finitely generated Banach $C(K)$-modules. It would be just a corollary of Theorem~\ref{t3} except for the fact that it contains two new equivalences that are not true in the general case.

 Recall (see~\cite{KO1}) that a Banach $C(K)$ module is called \textbf{finitely generated} if there are $x_1, \ldots , x_n \in X$ such that the linear span of cyclic subspaces $X(x_i), i= 1, \ldots , n$ is dense in $X$. The elements $x_1, \ldots , x_n$ are called \textbf{generators} of $X$.

 \begin{theorem} \label{t4} Let $X$ be a finitely generated Banach $C(K)$-module such that  $m$ is injective and $m(C(K))$ is weak operator closed in $L(X)$. Then the conditions (1) - (5) of Theorem~\ref{t3} are equivalent to the following two conditions.

 $(1A)$ $K$ is totally disconnected, $X$ is a Veksler module, and $X$ does not contain a copy of $\ell^\infty$.

 $(2A)$ $K$ is Stonian, $X$ is a Kaplansky module, and $X$ does not contain a copy of $\ell^\infty$.

 \end{theorem}

 \begin{proof} Equivalence of (1) - (5) follows from Theorem~\ref{t3}. It is obvious that
 $(2A) \Rightarrow (1A) \Rightarrow (1)$. Therefore the theorem will be proved if show that (2A) follows from conditions (1) - (5).

 Hence we can assume that $K$ is hyperstonian, $m : C(K) \rightarrow L(X)$ is ($w^\star$, weak-operator)-continuous, and the Boolean algebra $\mathcal{B}$ of all idempotents in $C(K)$ is Bade complete on $X$. Furthermore, no cyclic subspace of $X$ contains a copy of $\ell^\infty$. Under this conditions we have to show that $X$ does not contain a copy of $\ell^\infty$. We will prove it by induction on the number of generators of $X$. When $n = 1$, $X$ is cyclic and therefore does not contain a copy of $\ell^\infty$. Suppose that whenever $X$ has $n$ generators the conditions $(1) - (5)$ imply (2A). Let $X$ have a minimum of $n+1$ generators $\{x_0, x_1, \ldots , x_n\}$. Consider the submodule $Y$ of $X$ with the generators $\{x_1, \ldots , x_n\}$. The map $m$ generates in an obvious way the map $m_Y : C(K) \rightarrow L(Y)$. Clearly $\ker{m_Y}$ is $w^\star$-closed and therefore $C(K)/\ker{m_Y} \cong C(K_Y)$ where $K_Y$ is a clopen subset of $K$. Hence $C(K_Y)$ is a dual Banach lattice. Let $e_Y$ be the characteristic function of $K_Y$. Then $e_Y\mathcal{B}$ is Bade complete on $Y$. Therefore $m_Y(C(K_Y))$ is weak-operator closed in $L(Y)$ and $Y$ satisfies conditions (1) - (5). Then by the induction hypothesis no closed subspace of $Y$ is isomorphic to $\ell^\infty$. Next we consider the factor $X/Y$ with elements $[x] = x + Y, x \in X$. Because $Y$ is a $C(K)$-submodule of $X$ the map
 $m_{X/Y} : C(K) \rightarrow L(X/Y)$ is well defined. It is clear that $\ker{m_{X/Y}}$ is a $w^\star$-closed ideal in $C(K)$ and therefore $C(K)/\ker{m_{X/Y}} \cong C(K_0)$ where $K_0$ is a clopen subset of $K$. Moreover $C(K_0)$ is a dual Banach lattice. Notice that $X/Y = X/Y([x_0])$ is a cyclic Banach space and because the map $m_{X/Y} : C(K_0) \rightarrow L(X/Y)$ is ($w^\star$, weak-operator)-continuous we see that the Boolean algebra $e_0 \mathcal{B}$, where $e_0$ is the characteristic function of $K_0$, is Bade complete on $X/Y$. Hence by Theorem~\ref{t2} $X/Y$ has order continuous norm when represented as a Banach lattice. Then by Theorem~\ref{t1} $X/Y$ does not contain any copy of $\ell^\infty$. Since not containing $\ell^\infty$ is a three-space property~\cite{Ca}, $X$ does not contain a copy of $\ell^\infty$, and the proof is complete.

 \end{proof}

 \begin{remark} \label{r7} In Theorem~\ref{t4} we cannot dispense with the condition that $X$ is finitely generated. Indeed, $\ell^\infty$ considered as a $\mathds{C}$-module satisfies conditions (1) - (5) of Theorem~\ref{t3} but it does contain a copy of itself.

 \end{remark}

 Our final result relates to the dual Radon-Nikodym property and is a corollary of Theorem~\ref{t4} and~\cite[Theorem 3.4]{KO3}.

 \begin{theorem} \label{t5} Let $K$ be a totally disconnected compact space and $X$ be a finitely generated Veksler $C(K)$-module. Assume also that $m(C(K))$ is weak-operator closed in $L(X)$. Then the following conditions are equivalent.

 \begin{enumerate}
   \item $X^\prime$ has the Radon-Nikodym property.
   \item $X$ does not contain a copy of $\ell^1$.
   \item No cyclic subspace of $X$ contains a copy of $\ell^1$.
 \end{enumerate}

 \end{theorem}

 \begin{proof} Without loss of generality we may assume that $m$ is injective. The implications $(1) \Rightarrow (2) \Rightarrow (3)$ are trivial. Assume (3). Since $\ell^\infty$ contains a copy of $\ell^1$ no cyclic subspace of $X$ contains a copy of $\ell^\infty$. By Theorem~\ref{t4} the Boolean algebra $\mathcal{B}$ of all idempotents in $C(K)$ is Bade complete. Now Theorem 3.4 ($(3) \Rightarrow (1)$) in~\cite{KO3} implies that $X^\prime$ has the Radon - Nikodym property.

 \end{proof}

\end{document}